\documentclass[leqno]{amsart}
\usepackage{amsmath}
\usepackage{amssymb}
\usepackage{amsthm}
\usepackage{enumerate}
\usepackage[mathscr]{eucal}
\usepackage{graphicx}
\theoremstyle{plain}
\newtheorem{theorem}{Theorem}[section]
\newtheorem{prop}[theorem]{Proposition}

\theoremstyle{definition}
\newtheorem{defn}{Definition}[section]
\newtheorem{remark}{Remark}[section]

\usepackage[pagewise]{lineno}
\begin{document}
	
	\title[A complete characterization of polygonal Radon planes]{A complete characterization of polygonal Radon planes}
	\author[  Kalidas Mandal, Debmalya Sain and Kallol Paul ]{Kalidas Mandal, Debmalya Sain and Kallol Paul  }

	\newcommand{\acr}{\newline\indent}
	
	\address[Mandal]{Department of Mathematics\\ Jadavpur University\\ Kolkata 700032\\ West Bengal\\ INDIA}
	\email{kalidas.mandal14@gmail.com}
	
	\address[Sain]{Department of Mathematics\\ Indian Institute of Science\\ Bengaluru 560012\\ Karnataka \\India\\ }
	\email{saindebmalya@gmail.com}
	
	\address[Paul]{Department of Mathematics\\ Jadavpur University\\ Kolkata 700032\\ West Bengal\\ INDIA}
	\email{kalloldada@gmail.com}

	\thanks{ The first author would like to thank CSIR, Govt. of India for the financial support in the form of junior research fellowship. The research of Dr. Debmalya Sain is sponsored by Dr. D. S. Kothari Post-doctoral Fellowship. Dr. Debmalya Sain would also like to acknowledge the warm hospitality of his dearest friend Arijeet Roy Chowdhury and his wife Mandira Saha. The research of Prof Kallol Paul  is supported by project MATRICS  of DST, Govt. of India. } 
	

\begin{abstract}
We study the structure of the unit sphere of polygonal Radon planes from a geometric point of view. In particular, we prove that a $ 2 $-dimensional real polygonal Banach space $ \mathbb{X} $ \emph{cannot} be a Radon plane if the number of vertices of its unit sphere is $ 4n, $ for some $ n \in \mathbb{N}. $ We next obtain a complete characterization of polygonal Radon planes in terms of a tractable geometric concept introduced in this article. It follows from our characterization that every regular polygon with $ 4n+2 $ vertices, where $ n \in \mathbb{N}, $ is the unit sphere of a Radon plane. We further give example of a family of Radon planes for which the unit spheres are hexagons, but not regular ones. 
\end{abstract}

\maketitle
\section{Introduction.} 

The purpose of this short note is to study Radon planes, i.e., two-dimensional real Banach spaces where Birkhoff-James orthogonality is symmetric, whose unit spheres are polygons. Indeed, we obtain a complete characterization of polygonal Radon planes in terms of a novel geometric concept which is rather straightforward to verify. Let us first establish the relevant notations and the terminologies to be used in the paper. \\

Let $(\mathbb{X},\|.\|)$ be a two-dimensional real Banach space. Given any two elements $ u,v \in \mathbb{X}, $ let $ \overline{uv}=\{(1-t)u+tv~:t \in [0,1]\} $ denote the closed straight line segment joining $ u $ and $ v. $ Let $B_{\mathbb{X}} = \{x \in \mathbb{X} \colon \|x\| \leq 1\}$ and $S_{\mathbb{X}} = \{x \in \mathbb{X} \colon \|x\|=1\}$ be the unit ball and the unit sphere of $\mathbb{X}$, respectively. It is perhaps obvious that without loss of generality, we may assume $ \mathbb{X} $ to be $ \mathbb{R}^{2}, $ equipped with some appropriate norm. It is in this sense that we identify the zero vector with the origin $ (0,0) $ in $ \mathbb{R}^{2}. $ For any two elements  $x,y \in {\mathbb{X}}$, $x$ is said to be Birkhoff-James orthogonal to $y$ \cite{B}, written as $x \perp_B y$, if $ \|x+\lambda y\|\geq\|x\|$ for all $ \lambda \in \mathbb{R}.$ It is easy to observe that Birkhoff-James orthogonality is homogeneous, i.e., given any $ x,y \in \mathbb{X} $ and any two scalars $ \alpha,\beta \in \mathbb{R}, $ we have, $ x \perp_{B} y $ implies that $ \alpha x \perp_{B} \beta y. $  However, we would like to note that Birkhoff-James orthogonality is not necessarily symmetric in $ \mathbb{X}, $ i.e., given any two elements  $ x,y $ in $  \mathbb{X},$ $ x \perp_{B} y $ does not necessarily imply that $ y \perp_{B} x. $ For an element $x\in \mathbb{X},$ we say that $x$ is left symmetric \cite{S} if $x\perp_B y$ implies $y\perp_B x$ for any $y\in \mathbb{X}.$  $ \mathbb{X} $ is said to be a Radon plane \cite{R} if Birkhoff-James orthogonality is symmetric in $ \mathbb{X}. $ It is well-known \cite{J} that for a Banach space $ \mathbb{X} $ of dimension greater than or equal to $ 3, $ Birkhoff-James orthogonality is symmetric in $ \mathbb{X} $ if and only if the norm is induced by an inner product. However, it follows from the seminal works of James \cite{J} and Day \cite{D} that there exist two-dimensional real Banach spaces for which Birkhoff-James orthogonality is symmetric but the norm is not induced by an inner product. In this paper we study the symmetry of Birkhoff-James orthogonality for a special class of two-dimensional real Banach spaces. $ \mathbb{X} $ is said to be polygonal if $ B_{\mathbb{X}} $ contains only finitely many extreme points, or, equivalently, if $ S_{\mathbb{X}} $ is a polygon. We say that $ S_{\mathbb{X}} $ is a regular polygon if all the edges of $ S_{\mathbb{X}} $ have the same length with respect to the usual metric on $ \mathbb{R}^{2} $ and all interior angles are equal in measure.  In this paper, our purpose is to  explore the connection between $ \mathbb{X} $ being a Radon plane and the number of vertices of $ S_{\mathbb{X}}. $ In particular, we prove that if $ \mathbb{X} $ is a two-dimensional real Banach space such that $ S_{\mathbb{X}} $ is a polygon with $ 4n $ vertices, where $ n \in \mathbb{N}, $ then $ \mathbb{X} $ is never a  Radon plane. We next obtain a complete characterization of polygonal Radon planes in terms of a geometric condition which is very easy to verify. Let us mention the relevant definitions in this context.

\begin{defn} [\textbf{Translated vertex property of an edge}]
	Let $\mathbb{X}$ be a two-dimensional real polygonal Banach space. An edge $L$ of $ S_{\mathbb{X}} $ is said to satisfy translated vertex property (TVP) if  $ 0 \in x + L $ for some $x \in \mathbb{X} $ and $ x + L$ intersects $ S_{\mathbb{X}}$ at a pair of vertices $ \pm v $ of $ S_{\mathbb{X}}. $ In that case, $ \pm v $ are called the translated vertices of $ S_{\mathbb{X}} $ corresponding to the edge $ L. $   Equivalently, an edge $L$ satisfies TVP if and only if the kernel of the extreme supporting functional $f$ corresponding to $L$ (see Definition $1.3$ of \cite{SPBB}) contains a pair of vertices of $S_{\mathbb{X}}.$
	The space $\mathbb{X}$ is said to satisfy TVP if every edge of $S_{\mathbb{X}}$ satisfies TVP.
\end{defn}

\begin{defn} [\textbf{Translated edge  property of two adjacent edges}]
	Let $\mathbb{X}$ be a two-dimensional real polygonal Banach space and let $L, M$ be any two adjacent edges of $ S_{\mathbb{X}} $. We say that the consecutive edges $L$ and $M$ satisfy translated edge property (TEP) if and only if \\
	(i) both $L$ and $M$ satisfy TVP with translated vertices $\pm v$ and $\pm w$ respectively.\\
	(ii) either $\overline{vw} $ or $\overline{v(-w)}$ is an edge with the translated vertices $ \pm x$.\\
	The space $\mathbb{X}$ is said to satisfy TEP if every pair of adjacent edges of $S_{\mathbb{X}}$ satisfies the same.
\end{defn} 

The importance of these two geometric concepts, introduced in this article, becomes evident in course of our study. Indeed, the following result is the major highlight of our present study: \emph{a two-dimensional real polygonal Banach space $ \mathbb{X} $ is a Radon plane if and only if $ \mathbb{X} $ satisfies TEP.} Moreover, it is easy to deduce from this geometric characterization of polygonal Radon planes that every regular polygon with $ 4n+2 $ vertices, where $ n \in \mathbb{N}, $ is the unit sphere of a Radon plane. We further study the geometry of Radon planes whose unit spheres are polygons but not necessarily regular ones. Indeed, we give a concrete example of a family of Radon planes whose unit spheres are irregular hexagons.\\

The notion of smoothness in a Banach space is intimately connected with the present study. An element $x\in S_{\mathbb{X}}$ is said to be a smooth point if there is a unique hyperplane $H$ supporting $B_{\mathbb{X}}$ at $x.$ The space $\mathbb{X}$ is said to be smooth if every point of $ S_{\mathbb{X}} $ is a smooth point. We would like to remark that since in the context of our present study $ \mathbb{X} $ is two-dimensional, any supporting hyperplane to $ B_{\mathbb{X}} $  is simply a straight line that touches $ B_{\mathbb{X}} $ at some point(s) so that $ B_{\mathbb{X}} $ lies entirely in one of the closed half-planes determined by the said straight line. It is immediate that unit sphere of any two-dimensional polygonal real Banach space consists of extreme points and smooth points only, the extreme points being the vertices of the polygon.  At every smooth point on the unit sphere, the only supporting hyperplane is the edge of the polygon that contains the corresponding point. However, at every vertex of the unit sphere, there are infinitely many supporting hyperplanes that lie inside a cone generated by the two edges of the polygon meeting at that particular vertex of the polygon. Indeed, this simple geometric idea is at the heart of our approach towards completely characterizing Radon planes whose unit spheres are polygons.

\section{Main results}
Let us begin with two simple observations that relate smooth points and extreme points of the unit ball of a two-dimensional real polygonal Banach space and describe the Birkhoff-James orthogonality set of any unit vector. We omit the proofs as it is rather obvious for the first proposition and for the second proposition, the proof follows from  \cite[Th. 2.1]{SPM}.
\begin{prop}\label{prop}
	Let $\mathbb{X}$ be a two-dimensional real polygonal Banach space. Let $x\in S_{\mathbb{X}}$. Then $x$ is an extreme point of $B_{\mathbb{X}}$ if and only if $x$ is not a smooth point.

\end{prop}
\begin{prop}\label{prop-1}
	Let $ \mathbb{X} $ be a two-dimensional real Banach space and $ x \in S_{\mathbb{X}}.$ Then the set $  x^{\bot} =  \{ y \in \mathbb{X} : x \bot_B y \} = K \cup -K$, where $K$ is a normal cone in $ \mathbb{X}.$ 
\end{prop} 
Note that a subset $ K $ of $ \mathbb{X} $ is said to be a normal cone in $ \mathbb{X} $ if \\
$ (i)~ K + K \subset K, (ii)~ \alpha K \subset K $ for all $ \alpha \geq 0 $ and $ (iii)~ K \cap (-K) = \{\theta\}. $\\

Our next proposition is useful in studying polygonal radon planes.
\begin{prop}\label{prop-2}
	Let $\mathbb{X}$ be a two-dimensional real polygonal Banach space such that $\mathbb{X}$ is a Radon plane.  Let $L$ be an edge of $S_{\mathbb{X}}$ and $v$ be an extreme point of $S_{\mathbb{X}}$ such that $L\perp_B v$. Then $v\perp_B u$ if and only if $u\in \pm L$.
	
\end{prop}
\begin{proof}
	Let $L$ be an edge of $S_{\mathbb{X}}$ such that $L\perp_B v$, where $v$ be an extreme point of $S_{\mathbb{X}}$. Since $\mathbb{X}$ is a Radon plane, if $u\in \pm L$ then $v\perp_B u.$\\
	
	Conversely, let $L\perp_B v$ and $v\perp_B u.$ We want to show that $u\in \pm L.$
	Since $L\perp_B v$, there exist a linear functional $f\in S_{\mathbb{X}^*}$ such that $f(x)=1$ for all $x\in L$ and $f(v)=0.$ Again, since $\mathbb{X}$ is a Radon plane $ v\perp_B u$, we have $u\perp_B v.$ Then there exist a linear functional $g\in S_{\mathbb{X}^*}$ such that $g(u)=1$ and $g(v)=0.$ Therefore, $v\in ker g$. We also have $v\in ker f$. Since the dimension of $\mathbb{X}$ is two, we have $ker f= ker g.$ This shows that $f=\lambda g$. As $f, g \in S_{\mathbb{X}^*}$, we have $f=\pm g.$ Therefore, $ f(u) = \pm 1 $ and so  $u\in \pm L.$
	
\end{proof}

Our next result illustrates the importance of the extreme points of the unit ball in studying Radon planes. 
\begin{theorem}\label{extreme}
	Let $\mathbb{X}$ be a  two-dimensional real  polygonal Banach space. Then $\mathbb{X}$ is a Radon plane if and only if Birkhoff-James orthogonality is left symmetric at each extreme point of $B_{\mathbb{X}}$. 
	
\end{theorem}
\begin{proof}
	Let $\mathbb{X}$ be a Radon plane. Then Birkhoff-James orthogonality is symmetric at each point of $S_{\mathbb{X}}$. Therefore, Birkhoff-James orthogonality is left symmetric at each extreme point of $B_{\mathbb{X}}.$\\
	
	Conversely, suppose Birkhoff-James orthogonality is left  symmetric at each extreme point of $B_{\mathbb{X}}.$ Let $x, y \in S_{\mathbb{X}}$ and $x\perp_B y.$ We show that $y\perp_B x.$ If $x$ is an extreme point of $B_{\mathbb{X}}$ then we are done. In a real  two-dimensional polygonal space, we know that each point $ S_{\mathbb{X}}$  is either an extreme point of $B_{\mathbb{X}}$ or a smooth point. So we consider  $x$ to be a smooth point.   Clearly $x$ can be written as $x= (1-t)v_1 +tv_2$ for some extreme points $v_1, v_2$ of $B_{\mathbb{X}}.$ As $x\perp_B y$ , there exists a linear functional $f\in S_{\mathbb{X}^*}$ such that $f(x)=1$ and $f(y)=0.$ Then we have, $(1-t)f(v_1)+tf(v_2)=f(x)=1\Rightarrow f(v_1)=f(v_2)=1.$  This shows that $v_1\perp_B y$ and $v_2\perp_B y.$ Again, by the given condition, we have $y\perp_B v_1$ and $y\perp_B v_2.$ Now, by Proposition \ref{prop-1}, the set $ x^{\bot} = \{ v \in \mathbb{X} : y \bot_B v \} = K \cup (-K),$ where $K$ is a normal cone in  $\mathbb{X}$, so it follows  that $y\perp_B (1-t)v_1 +tv_2=x.$ Thus $x \bot_B y $ implies that $ y \bot_B x$ so that Birkhoff-James orthogonality is symmetric. This completes the proof.
\end{proof}

If $ \mathbb{X} $ is a two-dimensional real polygonal Banach space then the number of vertices of $ S_{\mathbb{X}} $ is either $ 4n $ or $ 4n+2, $ for some $ n \in \mathbb{N}. $ The major aim of the present paper is to show that in the first case $ \mathbb{X} $ is \emph{never} a Radon plane, while in the second case, $ \mathbb{X} $ is a Radon plane provided some additional conditions are satisfied. We begin by proving the following result that addresses our first query.
\begin{theorem}\label{4n-gon}
	Let $\mathbb{X}$ be a two-dimensional real Banach space such that $ S_{\mathbb{X}}$ is a polygon with $4n$ vertices, where $n\in\mathbb{N}$. Then $\mathbb{X}$ is not a Radon plane.
\end{theorem}   
\begin{proof}
	If possible suppose $\mathbb{X}$ is a Radon plane. Without loss of generality, we take an edge $L$ of $ S_{\mathbb{X}}.$ Let  $f\in S_{\mathbb{X}^*}$ be the  extreme supporting linear functional corresponding to the edge $L$. Then $f(L) = 1$ and $f(v) = 0$ for some  $v \in S_{\mathbb{X}}.$   Clearly  $L\perp_B v$. Since $\mathbb{X}$ is a Radon plane, we have $v\perp_B L.$ Since $L$ is an edge, Birkhoff-James  orthogonality at $v$ is not right unique and so  it follows from \cite[Th. 4.1]{Ja} that $v$ is not a smooth point of $\mathbb{X}.$ Then using Proposition \ref{prop}, we have $v$ is an extreme point of $B_{\mathbb{X}}$ and so  $v$ is a  vertex of $S_{\mathbb{X}}.$  Now, let $u$ be the midpoint of $L$. Since $v\perp_B u$, there exist a supporting functional $g\in S_{\mathbb{X}^*}$ at $v$ such that $g(v)=1$ and $g(u)=0.$ The situation is illustrated in following figure 1:\\
	
	\begin{figure}[ht]
		\centering 
		\includegraphics[width=0.5\linewidth]{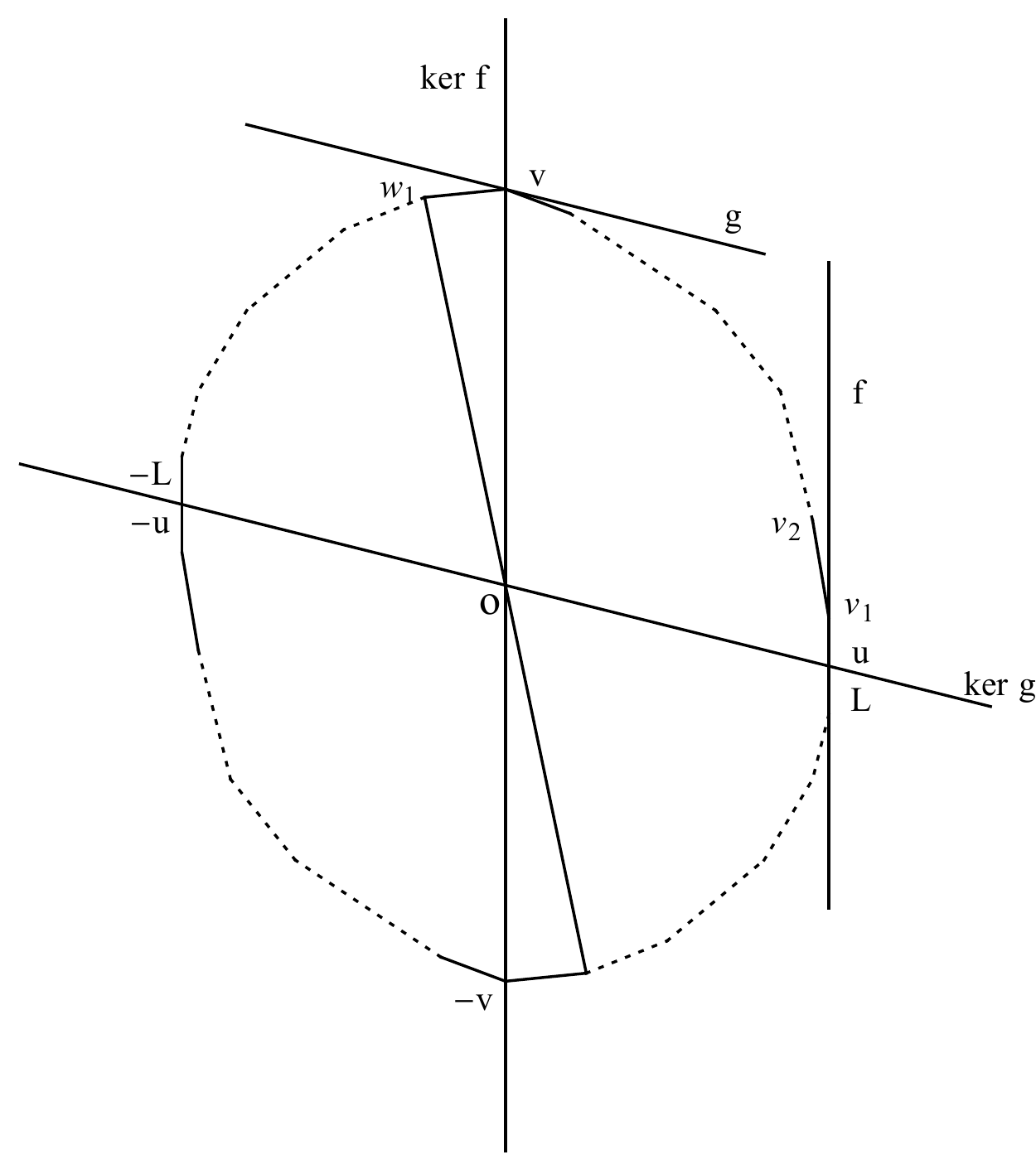}
		\caption{}
		\label{Figure 1}
	\end{figure}

	\noindent Then the plane $\mathbb{X}$ is divided  into the four closed sets: 
	$$ S_1=\{x\in \mathbb{X}: f(x)\geq 0, g(x)\geq 0\}, S_2=\{x\in \mathbb{X}: f(x)\leq 0, g(x)\geq 0\},$$
	$$ S_3=\{x\in \mathbb{X}: f(x)\leq 0, g(x)\leq 0\}, S_4=\{x\in \mathbb{X}: f(x)\geq 0, g(x)\leq 0\}.$$
	We note that $ker f$ contains two vertices of $S_{\mathbb{X}}$, namely, $ \pm v$ and $ker g$ contains two smooth points of $S_{\mathbb{X}},$ namely, $ \pm u.$ We claim that each closed set $S_i, i=1,2,3,4$ (excluding $ker f$) contains exactly the same number of vertices of $S_{\mathbb{X}}.$ Let $k_0$ denote the number of vertices of $S_{\mathbb{X}}$ in $S_1$ (excluding $\pm v$). To establish our claim, let us consider two adjacent vertices  $v_1$ and $v_2$ (may not be equal to $v$) of $S_{\mathbb{X}}$ in $S_1$ such that $v_1 \in S_{\mathbb{X}} \cap L.$ Then because of convexity of $ B_{\mathbb{X}},$ $v_2$ can be written as $v_2=(1-t)v_1 +t(\alpha v)$ for some  $\alpha>0$ and $t\in (0,1).$ This shows that $f(v_2)<f(v_1).$ Similarly, $v_2$ can written as $v_2=(1-s)v +s (\beta u)$ for some  $\beta>0$ and $s\in (0,1).$ This shows that $g(v_2)<g(v)$. Proceeding in the same way we can show that $g(v_1)<g(v_2).$ Now, $\overline {v_1v_2}$ is an edge of  $S_{\mathbb{X}}$ lying  in $S_1$ and $\overline{v_1v_2}\perp_B {v_2-v_1}.$ Let $w_1 =  \frac{v_2-v_1}{\|v_2-v_1\|}. $  Since $\mathbb{X}$ is a Radon plane, we have $ w_1 \perp_B \overline{v_1v_2}.$ As before Birkhoff-James orthogonality is not right unique at $w_1$ and so $w_1$ is a vertex of $S_{\mathbb{X}}.$ Since $f(w_1)<0$ and $g(w_1)>0$, we have  $w_1$ lies in $S_2$ (excluding $ker f$ and $ker g$). 
	Continuing this process for every edge in $S_1$, we may and do conclude that corresponding to every edge of $S_{\mathbb{X}}$ lying entirely in $S_1$, there is a vertex of $ S_{\mathbb{X}} $ in $S_2$ (excluding $ker f$). Therefore, the number of vertices of $ S_{\mathbb{X}} $ in $S_2$ (excluding $ker f$ ) is greater than or equal to $ k_0. $ Similarly, corresponding to every edge of $S_{\mathbb{X}}$ lying entirely in $S_2$, there is a vertex of $ S_{\mathbb{X}} $ in $S_1$ (excluding $ker f$). So the number of vertices of $S_{\mathbb{X}}$ in $S_1$ (excluding $ker f$) is greater than or equal to the number of vertices of $S_{\mathbb{X}}$ in $S_2$ (excluding $ker f$). This proves that $ S_{\mathbb{X}} $ has the same number of vertices in $S_1$ and  $S_2$ (excluding $ker f$). Moreover, by symmetry of $ S_{\mathbb{X}} $ about the origin, the number of vertices of $ S_{\mathbb{X}} $ in $S_1$ and $S_3$ ($S_2$ and $S_4$) must be equal. Therefore, the number of vertices of $ S_{\mathbb{X}} $ in each $S_i, i=1,2,3,4$ (excluding $ker f$) is  $ k_0. $ This completes the proof of our claim. Now, we count the total number of vertices of $S_{\mathbb{X}}$. Clearly, this turns out to be $4k_0+2$ ($k_0$ in each closed set $S_i, i=1,2,3,4$ except $ker f$ and 2 in  $ker f$). However, this clearly contradicts the fact that $S_{\mathbb{X}}$ is a polygon with $4n $ vertices. This completes the proof.
\end{proof}

We next obtain a complete characterization of polygonal Radon planes in terms of the translated edge property (TEP). 
\begin{theorem}\label{Radon}
	Let $\mathbb{X}$ be a  two-dimensional real polygonal Banach space. Then $\mathbb{X}$ is a Radon plane if and only if  $\mathbb{X} $ satisfies TEP.	
\end{theorem}
\begin{proof} We first prove the necessary part. 	
	Let $L$ be an edge of $S_{\mathbb{X}}.$ Let  $f\in S_{\mathbb{X}^*}$ be the  extreme supporting linear functional corresponding to the edge $L$. Then $f(L) = 1$ and $f(v) = 0$ for some  $v \in S_{\mathbb{X}}.$   Clearly  $L\perp_B v$. Since $\mathbb{X}$ is a Radon plane, we have $v\perp_B L.$ Since $L$ is an edge, Birkhoff-James orthogonality is not right unique and so  it follows from \cite[Th. 4.1]{Ja} that $v$ is not a smooth point of $\mathbb{X}.$ Again, it follows from Proposition \ref{prop} that $v$ is an extreme point of $B_{\mathbb{X}}$ and so $v$ is a  vertex of $S_{\mathbb{X}}.$ This shows that every edge satisfies the TVP.  Let 
	$L$ and $M$  be any two adjacent edges of $S_{\mathbb{X}}$ with a common vertex $u$. Also, let $f$ and $g$ be the extreme supporting functionals corresponding to the edges $L$ and $M$, respectively so that  $f(u)=g(u)=1. $ Also by the TVP, we get $ker f\cap S_{\mathbb{X}}=\{\pm v\}$ and  $ker g\cap S_{\mathbb{X}}=\{\pm w\}$ for some extreme points $v$ and $w$ of $B_{\mathbb{X}}.$ Figure 2, given bellow, illustrates the situation.
	
		\begin{figure}[ht]
		\centering 
		\includegraphics[width=0.5\linewidth]{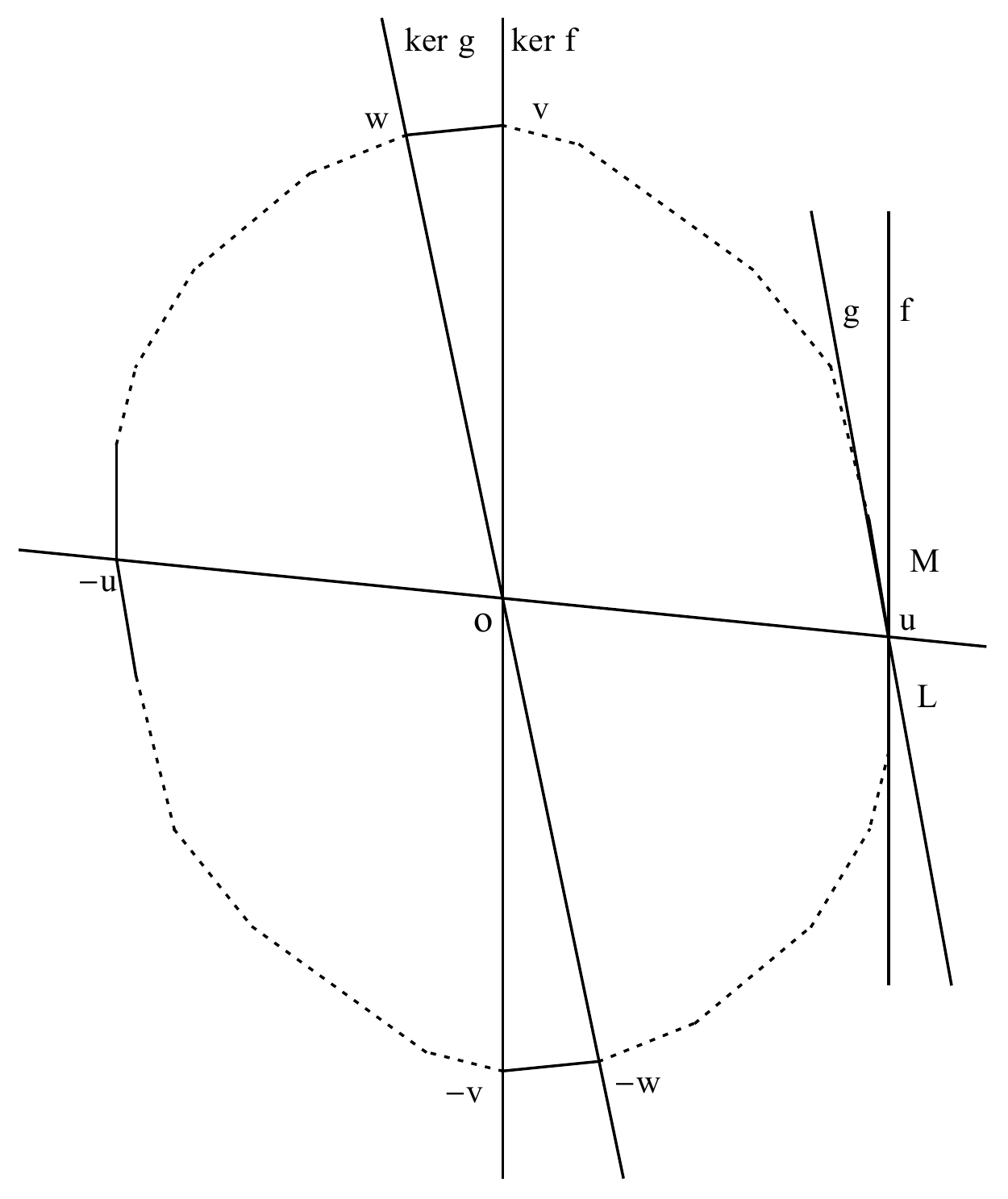}
		\caption{}
		\label{Figure 2}
	\end{figure}
	\noindent We show that $ \overline {vw}$ or $ \overline {v(-w)}$ is an edge of $S_{\mathbb{X}}.$ Clearly $ L \bot_B v, M \bot_B w$ and because of symmetricity of Birkhoff-James orthogonality we get $ v \bot_B L, w \bot_B M.$ Now $u,v,w $ are vertices with the property that $ u \bot_B v $ and $ u \bot_B w$.  Let $y_s =  (1-s)v + sw $ and $ z_s = (1-s)v - sw, $ where $ s \in (0,1).$ We claim that  $ u \bot_B y_s $  for all $ s \in (0,1) $ or   $u \bot_B z_s $ for all $ s \in (0,1).$ It follows from Proposition \ref{prop-1}, the set $ u^{\bot} = \{ x \in \mathbb{X} : u \bot_B x \} = K \cup (-K),$ where $K$ is a normal cone in  $\mathbb{X}$ and so the claim is established. We assume that $ u \bot_B y_s $  for all $ s \in (0,1) .$ Then $ y_s \bot_B u $ for all $ s \in (0,1),$ i.e., $ \overline{vw} \bot_B u.$  Clearly, $ \| y_s \| \leq 1$  for all $ s \in (0,1) .$ If $ \| y_s\| = 1 $ for all $ s \in (0,1) $ then $ \overline{vw}$ is an edge and we are done. 
	If possible let $\overline{vw}$ is not an edge then we can find vertices $v_1,w_1$, not necessarily unequal, so that $ \overline{vv_1} $ and $ \overline{ww_1}$ are edges and proceeding as before we can conclude that  $ \overline{vv_1} \bot_B u,  \overline{ww_1} \bot_B u. $ This is not possible and so $\overline{vw}$ is an edge. This completes the proof of necessary part.\\

	We next prove the sufficient  part using Theorem \ref{extreme}.  So it is sufficient to show that  $x\perp_B y$ implies $y\perp_B x,$   whenever  $ x $ is an extreme point of $ B_{\mathbb{X}}$ and $y \in S_{\mathbb{X}}.$  Let $L$ and $M$ be the two edges of $S_{\mathbb{X}}$ containing $x$.  Let $f,g\in S_{\mathbb{X}^*}$ be the extreme supporting  linear functionals corresponding to edges $L$ and $M$  such that $ker f\cap S_{\mathbb{X}}=\{\pm v\}$ and  $ker g\cap S_{\mathbb{X}}=\{\pm w\}.$  Then $ \overline {vw}$ or $ \overline {v(-w)}$ is an edge of $S_{\mathbb{X}}.$ Without loss of generality we may assume that $ \overline {vw}$ is an edge of $S_{\mathbb{X}}.$ Then by the TVP, we get  $\overline {vw}\perp_B x.$  We want to show that $y\in\overline {vw}.$ Following the proof of \cite[Th. 2.2]{SPBB} we note that every supporting functional at $x$ is a convex combination of $f$ and $g$. As $ x \bot_B y $ so there is some $t \in [0,1] $ such that $ [(1-t)f + tg] (x) = 1 $ and $ [(1-t)f + tg](y) = 0.$ This implies that $ f(y)g(y) < 0 $ and so either $ f(y) > 0, g(y) < 0$ or $ f(y) < 0, g(y) > 0.$ As $ \|y\| =1 $ we can conclude that either $ y \in   \overline {vw} $ or $ -y \in   \overline {vw}. $ Thus we get $ y \bot_B x.$ This completes the proof. 	
\end{proof}

As an application of the above theorem, it is easy to observe that regular polygons with $ 4n+2 $ number of vertices, where $ n \in \mathbb{N}, $ are the unit spheres of Radon planes.

\begin{theorem}\label{4n+2-gon}
	Let $\mathbb{X}$ be a two-dimensional real Banach space such that $S_{\mathbb{X}}$ is a regular polygon with $4n+2$ vertices, where $n\in \mathbb{N}$. Then $\mathbb{X}$ is a Radon plane.
\end{theorem}
\begin{proof}
	We show that $\mathbb{X}$ satisfies the TEP. Let $L$ and $M$ be two adjacent edges  of $S_{\mathbb{X}}$ with common vertex $v_1.$ Then $-L$ is also an edge of $S_{\mathbb{X}}$, parallel to $L$. We choose the $ X $-axis to be the line through the origin that bisects both $ L $ and $ -L. $ Clearly, there is no vertex of $S_{\mathbb{X}}$ on the $X$-axis. We know that for a regular polygon with $2n$  vertices, the line of symmetries are the following straight lines:\\
	
	\noindent(i) The straight line passing through the midpoints of any two parallel edges of the polygon.\\
	(ii) The straight line passing through any two opposite vertices $v_i$ and $-v_i$ of the polygon.
	
	\begin{figure}[ht]
\centering 
\includegraphics[width=0.5\linewidth]{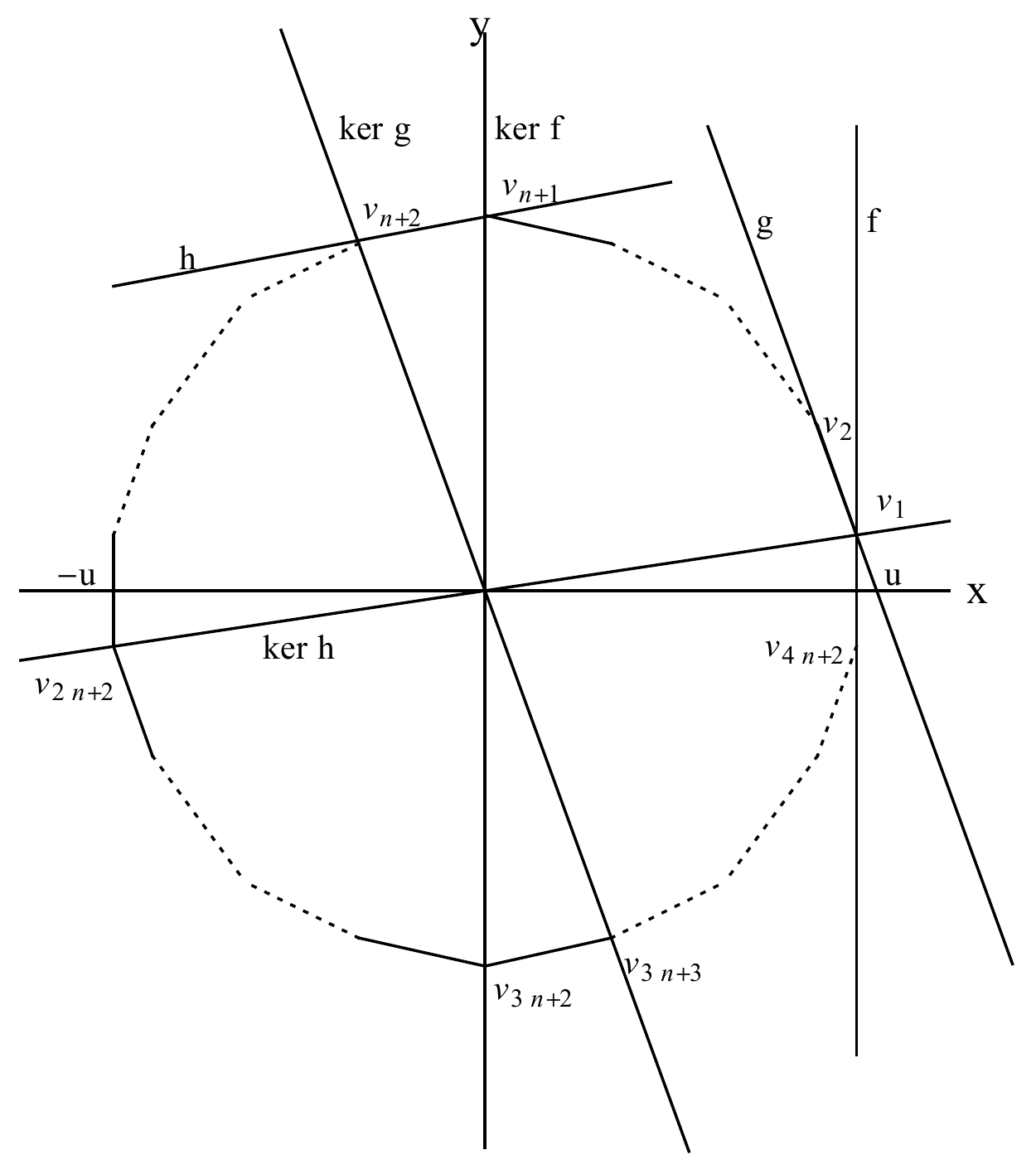}
\caption{}
\label{Figure 3}
\end{figure}

	We choose the $ Y $-axis to be the straight line through the origin, which is parallel to $L$ and $-L.$  Because of regularity of $S_{\mathbb{X}}$ it is easy to see that both the $X$ and $Y$-axes are line of symmetry. Let the $Y$-axis meets $S_{\mathbb{X}}$ at two points $ \pm w = \pm (0, \beta), $ where $ \beta \in \mathbb{R}. $  The two axes divide the plane into four quadrants. By symmetry of $S_{\mathbb{X}}$ about the two axes, each quadrant (excluding the two axes) contains exactly same number of vertices. Since there is no vertex of $S_{\mathbb{X}}$ on the $X$-axis and the number of vertices of $S_{\mathbb{X}}$ is $4n+2$, we conclude that $ \pm w $ are two vertices of $S_{\mathbb{X}}.$ Let  the vertex $v_1$ be  in the first quadrant. The other vertices of $ S_{\mathbb{X}} $ are chosen anticlockwise and are denoted by $v_2, v_3,.......,v_{4n+2}$, where $v_j=(cos\frac{(2j-1)\pi}{4n+2},sin\frac{(2j-1)\pi}{4n+2}), j=1,2,....,4n+2.$ We observe that the following  holds:\\ 
	\noindent (1) The other vertex on $L$ is $v_{4n+2}$.\\
	(2) $v_i$ and $v_{i+1}$ are adjacent, $i=1,2,.....,4n+1$. \\
	(3) $w=v_{n+1}$ and $-w=v_{3n+2}$.\\
	(4) The vertices of $S_{\mathbb{X}}$ on $-L$ are $v_{2n+1}$ and $v_{2n+2}$ in the third and fourth quadrant, respectively.

\noindent Without loss of generality we take the two adjacent edges $L=\overline{v_1 v_{4n+2}}$ and $M=\overline{v_1v_2}$ of $S_{\mathbb{X}}.$ Let $f,g\in S_{\mathbb{X}^*}$ be the extreme supporting  linear functionals at $v_1$ corresponding to edges $L$ and $M$ respectively such that $f(v_1)=g(v_1)=1$. It is easy to see that  $ker f\cap S_{\mathbb{X}}=\{v_{n+1},v_{3n+2}\}$ and  $ker g\cap S_{\mathbb{X}}=\{v_{n+2},v_{3n+3}\}.$  Then $ \overline {v_{n+1}v_{n+2}}$ and $ \overline {v_{3n+2}v_{3n+3}}$ are the edges of $S_{\mathbb{X}}$ and $\overline {v_{n+1}v_{n+2}}=- \overline {v_{3n+2}v_{3n+3}}.$ Now, let $h\in S_{\mathbb{X}^*}$ be the extreme supporting  linear functional corresponding to edge $\overline {v_{n+1}v_{n+2}}$ such that $h(x)=1$, where $x\in \overline {v_{n+1}v_{n+2}}.$ It is easy to verify that $ker h\cap S_{\mathbb{X}}=\{v_1,v_{2n+2}\}$ and $v_1=-v_{2n+2}.$ Thus the edges $L$ and $M$ satisfy the TEP. So using Theorem \ref{Radon}, we conclude that $\mathbb{X}$ is a Radon plane.
	
\end{proof}

 Combining Theorem \ref{4n-gon} and Theorem \ref{4n+2-gon}, we have the following complete characterization of Radon planes having regular polygons as unit spheres.
 \begin{theorem}\label{characterization}
 Let $\mathbb{X}$ be a two-dimensional real Banach space such that $S_{\mathbb{X}}$
 is a regular polygon. Then $\mathbb{X}$ is a Radon plane if and only if the number of vertices of $S_{\mathbb{X}}$ is $4n+2$, where $n\in \mathbb{N}$.	
 \end{theorem}
\begin{remark}\label{Heil}
We would like to mention that Heil \cite{H} was the first to  observe that if $\mathbb{X}$ is a real two-dimensional polygonal Banach space such $S_{\mathbb{X}}$ is a regular polygon then $\mathbb{X}$ is a Radon plane if and only if the number of vertices of $ S_{\mathbb{X}} $ is $ 4n +2$ for some natural number $n.$ In this paper, we have provided an alternative geometric approach to prove the same. Furthermore, our results in this direction generalize the relevant observations made by Heil in \cite{H} by taking into consideration all possible convex polygons, and not just regular ones.  
\end{remark}

Continuing in the spirit of Remark \ref{Heil} We next present a family of Radon planes whose unit spheres are polygons, but not \emph{regular} polygons.

\begin{theorem}\label{hexagon}
Let $\mathbb{X}=\mathbb{R}^2$ be endowed with a norm so that $S_{\mathbb{X}}$ is a hexagon with four vertices $\pm(1,\alpha),\pm(1,-\alpha)$, where $ \alpha > 0. $ Then $\mathbb{X}$ is a Radon plane if and only if the other two vertices of $S_{\mathbb{X}}$ are $\pm(0,2\alpha)$ or $\pm(2,0).$
\end{theorem}
\begin{proof}
Let us first prove the sufficient part of the theorem. Suppose $\mathbb{X}$ is a two-dimensional real Banach space, whose unit sphere $S_{\mathbb{X}}$ is given by the hexagon with the vertices $\pm(1,\alpha),\pm(1,-\alpha),\pm(0,2\alpha)$. The scenario is illustrated in Figure $ 4. $ Let us first denote the edge between the vertices $(1,\alpha)$ and $ (1,-\alpha) $ by $L_1$. Then $-L_1$ is also an edge of $S_{\mathbb{X}}$ between the vertices $(-1,-\alpha)$ and $(-1,\alpha)$. Since any point on $L_1\bigcup (-L_1),$ except the extreme points $(1,\alpha), (1,-\alpha),(-1,-\alpha)$ and $(-1,\alpha),$ is smooth, the supporting line to $ B_{\mathbb{X}} $ at any point on $L_1$ and $-L_1,$ except the four vertices,  is the straight line containing $L_1$ and $-L_1$, respectively. Moreover, for the two vertices on $ L_1 (-L_1), $ the straight line containing $ L_1 (-L_1) $ is one of the supporting lines to $ B_{\mathbb{X}} $ at the respective vertex. Now, let us consider a straight line through origin, which is parallel to $L_1$ and $-L_1$. This straight line will meet $S_{\mathbb{X}}$ at $(0,2\alpha)$ and $(0,-2\alpha)$. Therefore, $u\perp_B \pm(0,2\alpha)$ for all $u\in L_1\bigcup (-L_1)$. \\

\begin{figure}[ht]
\centering 
\includegraphics[width=0.4\linewidth]{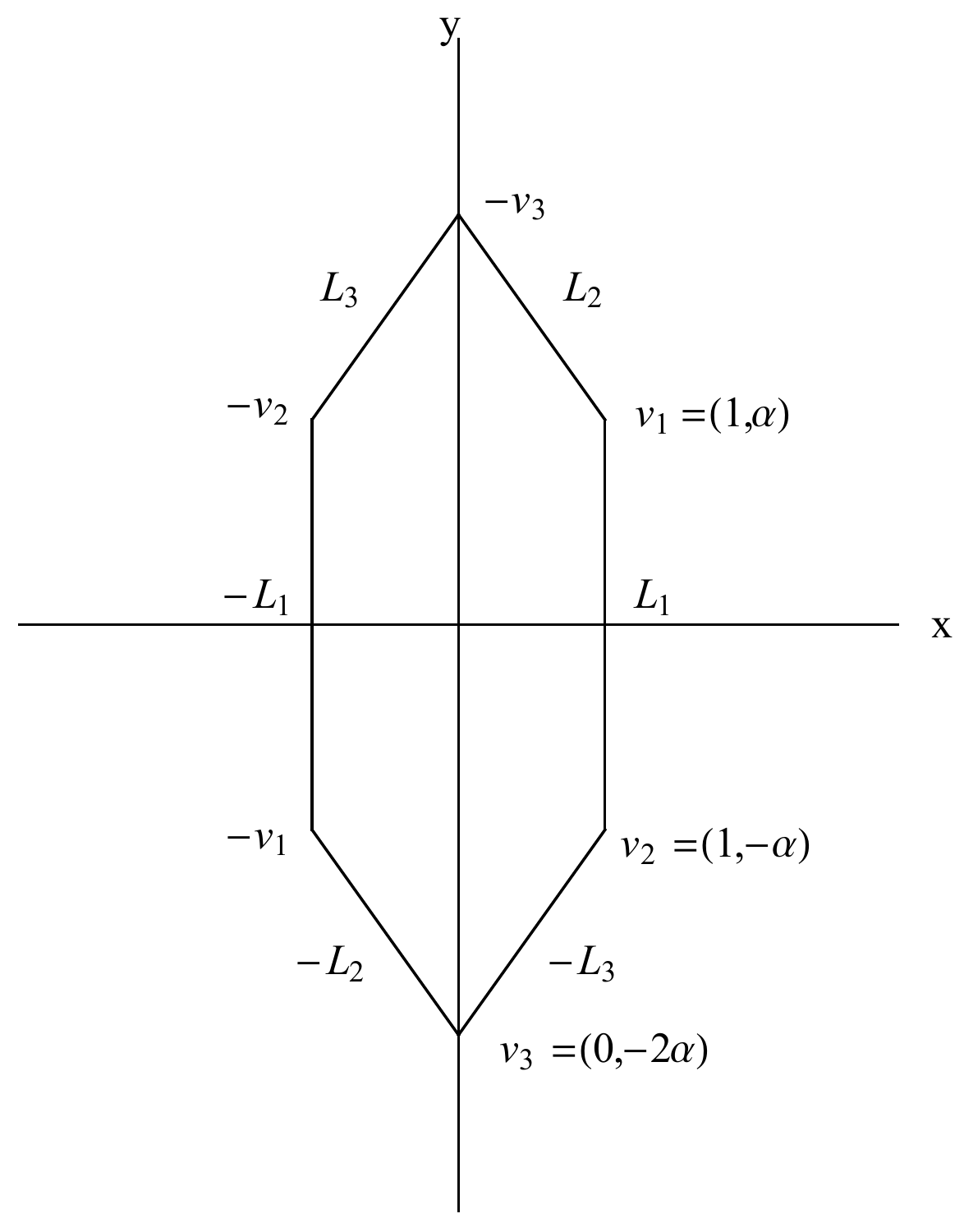}
\caption{}
\label{Figure 4}
\end{figure}

Next, let us consider the edge between the vertices $(1,\alpha)$ and $(0,2\alpha)$ and the edge between the vertices $(0,2\alpha)$ and $(-1,\alpha)$. Let us denote these edges by $L_2$ and $L_3$, respectively. Then $-L_2$ is an edge between the vertices $(-1,-\alpha)$ and $(0,-2\alpha)$ and $-L_3$ is an edge between the vertices $(0,-2\alpha)$ and $ (1,-\alpha)$. It is easy to verify that the following statements hold:\\

\noindent (1) $L_2$, $-L_2$ is parallel to the straight line joining the two points $(-1,\alpha),(1,-\alpha)$. \\
(2) $L_3,-L_3$ is parallel to the straight line joining the two points $(1,\alpha),(-1,-\alpha)$. \\

Now, since $(0,2\alpha)$ is an extreme point of $B_{\mathbb{X}}$, it follows from Proposition \ref{prop} that $(0,2\alpha)$ is non-smooth point. So there are infinitely many supporting lines to $B_{\mathbb{X}}$ at $(0,2\alpha)$. We observe that any supporting line to $B_{\mathbb{X}}$ at $(0,2\alpha)$ entirely lies inside the cone generated by the following two straight lines:\\

\noindent (i) The straight line containing the edge $L_2$.\\
(ii) The straight line containing the edge$L_3$.\\
 
Therefore, the straight line through the origin which is parallel to any supporting line to $B_{\mathbb{X}}$ at $(0,2\alpha)$ lies entirely inside the cone generated by the two straight lines containing the points $(-1,\alpha),(1,-\alpha)$ and $(1,\alpha),(-1,-\alpha)$. 
Similarly, we can show that the straight line through the origin which is parallel to any supporting line to $ B_{\mathbb{X}} $ at $(0,-2\alpha)$ lies entirely inside the cone generated by the two straight lines containing the points $(-1,\alpha),(1,-\alpha)$ and $(1,\alpha),(-1,-\alpha)$. Therefore, $\pm(0,2\alpha)\perp_B u$ for all $u\in L_1\bigcup (-L_1)$.\\
 
By using similar arguments, we can show that the following statements hold true:\\

\noindent (a) $v\perp_B \pm(-1,\alpha)$ for all $v\in L_2 \bigcup (-L_2)$ and $\pm(-1,\alpha)\perp_B v$ for all $v\in L_2 \bigcup (-L_2)$,\\
(b) $w\perp_B \pm(1,\alpha)$ for all $w\in L_3\bigcup (-L_3)$ and $\pm(1,\alpha)\perp_B w$ for all $w\in L_3 \bigcup (-L_3)$.\\

Therefore, using the homogeneity property of Birkhoff-James orthogonality, it follows that Birkhoff-James orthogonality is symmetric in $ \mathbb{X}. $ In other words, $ \mathbb{X} $ is a Radon plane.We would like to note that the proof can be completed using similar arguments if the other two vertices of $ S_{\mathbb{X}} $ are $ \pm(2,0). $This completes the proof of the sufficient part.\\

Let us now prove the necessary part of the theorem. Let $\mathbb{X}$ be a two-dimensional real Banach space such that $S_{\mathbb{X}}$ is a hexagon with two vertices $v_1=(1,\alpha),v_2=(1,-\alpha)$. Then $-v_1=(-1,-\alpha),-v_2=(-1,\alpha)$ are also vertices of $S_{\mathbb{X}}$. Suppose $ \mathbb{X} $ is a Radon plane. Now, one of the following two cases must hold.\\

\noindent (1) The $ X $-axis contains no vertices of $ S_{\mathbb{X}}. $\\
(2) The $ X $-axis contains two vertices of $ S_{\mathbb{X}}. $\\
 
First we assume that there are no vertices of $S_{\mathbb{X}}$ on the $X$-axis. 
Let us denote the edge between the vertices $v_1$ and $v_2$ by $L$. Then $-L$ is the edge between the verices $-v_1$ and $-v_2$. Now, it is clear that $X$-axis meets $L$ at the point $(1,0)$ and $-L$ at the point $(-1,0)$, which are smooth point of $S_{\mathbb{X}}$.  $Y$-axis meets $S_{\mathbb{X}}$ at two points $ \pm v = \pm (0, \beta), $ where $ \beta \in \mathbb{R}. $ It is easy to observe that given any $ y \in L, $ we have, $ y \perp_B v. $ Since Birkhoff-James orthogonality is symmetric in $ \mathbb{X} $, we must have that $ v \perp_B y $ for any $ y \in L. $  This shows that Birkhoff-James orthogonality is not right unique at $v$ and so  it follows from \cite[Th. 4.1]{Ja} that $ v $ must be a non-smooth point of $ \mathbb{X}. $ Therefore, the other two vertices of $S_{\mathbb{X}}$ is of the form $\pm(0,\beta)$. We denote the vertex $(0,-\beta)$ by $v_3$. Then the vertex $(0,\beta)$ is $-v_3$. Again, by the symmetry of Birkhoff-James orthogonality in $ \mathbb{X}, $ we must have that the edges  $\overline{v_2v_3}$ and $\overline {v_3(-v_1)}$ are parallel to $\overline{v_1(-v_1)}$ and $ \overline {v_2(-v_2)}$, respectively. It is easy to see that this is true only when $\beta=2\alpha$.
Now, we assume that $X$-axis contains two vertices of $S_{\mathbb{X}}.$ Then the other two vertices of $S_{\mathbb{X}}$ is of the form $w=\pm(\eta, 0)$, where $\eta \in \mathbb{X}$. The corresponding picture of the unit ball is illustrated in Figure $ 5. $ The rest of the proof of the fact that $ \eta=2 $ in this case can be completed similarly as above.
 
\begin{figure}[h]
\centering 
\includegraphics[width=0.4\textwidth]{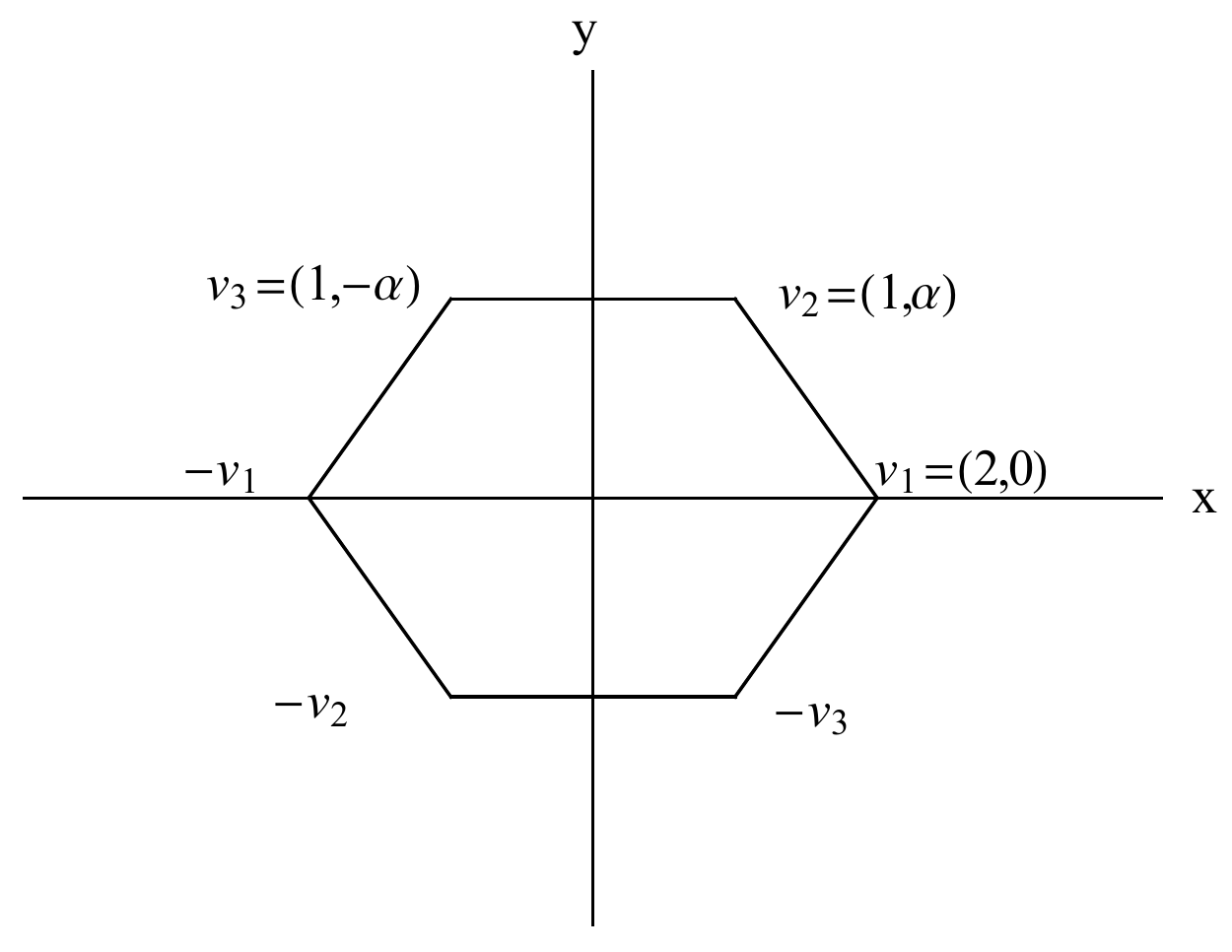}
\caption{}
\label{Figure 5}
\end{figure}

This completes the proof of the necessary part of the theorem and establishes the theorem completely.     
\end{proof}

\begin{remark}
If the vertices of $ S_{\mathbb{X}} $ in the above theorem are $\pm(1,\alpha), \pm(1,-\alpha)$ and $\pm(0,2\alpha)$ then $\mathbb{X}$ is a Radon plane with regular polygonal unit ball if and only if $\alpha= \frac{1}{\sqrt {3}}$. On the other hand, if the vertices of $ S_{\mathbb{X}}$ are $\pm(1,\alpha), \pm(1,-\alpha)$ and $\pm(2,0)$ then $\mathbb{X}$ is a Radon plane  with regular polygonal unit ball if and only if $\alpha= \sqrt{3}$. In all other cases, $ \mathbb{X} $ is a Radon plane whose unit ball is an irregular polygon.
\end{remark}
\textbf{Conclusion:}
In this paper we have obtained a complete geometric characterization of polygonal Radon planes. Furthermore, we have given concrete examples of Radon planes for which the unit spheres are polygons, but not regular polygons. Motivated by these findings, we end the present paper with the following:\\

\textbf{Open question:} \textit{Given any $ n \in \mathbb{N}, $ does there exist a Radon plane $ \mathbb{X} $ such that $ S_{\mathbb{X}} $ is an irregular polygon with $ 4n+2 $ vertices? If the answer to this question is in the affirmative, then find a generalized algorithm to construct Radon planes whose unit spheres are irregular polygons with $ 4n+2 $ vertices.}\\

\end{document}